\documentclass{personalart}

\usepackage{personalart}

\title{Homogeneous division polynomials for Weierstrass elliptic curves}
\author{Jinbi Jin \\ Universiteit Leiden \\ Mathematisch Instituut \\ Niels Bohrweg 1 \\ 2333 CA Leiden \\ jjin@math.leidenuniv.nl\do}
\date{\today}

\begin{document}
  \maketitle


\subsubsection*{Abstract}
Starting from the classical division polynomials we construct
homogeneous polynomials $\alpha_n$, $\beta_n$, $\gamma_n$ such that for $P = (x:y:z)$ on an elliptic curve in Weierstrass form over an arbitrary ring we have $nP = \pth*[big]{\alpha_n(P):\beta_n(P):\gamma_n(P)}$.
To show that $\alpha_n,\beta_n,\gamma_n$ indeed have this property we use the a priori existence of such polynomials, which we deduce from the Theorem of the Cube.

We then use this result to show that the equations defining the modular curve $Y_1(n)_{\bb C}$ computed for example by Baaziz, in fact are equations of $Y_1(n)$ over $\bb Z[1/n]$.

\subsubsection*{Acknowledgements}
This work is partially funded by NWO, the Netherlands Organisation for Scientific Research.
This work was originally a part of my MSc thesis, and I would like to thank Bas Edixhoven for his supervision at the time, especially for mentioning the Theorem of the Cube.
I would also like to thank Bas Edixhoven, Andreas Enge, Lenny Taelman, and a referee for their comments on previous versions of this paper.

\section*{Introduction}

The main problem treated in this paper is the following.

\begin{center}
  {\em Given an integer $n$, construct a triple of homogeneous polynomials that defines multiplication by $n$ on all {\em Weierstrass curves} (i.e.~projective plane curves given by a Weierstrass polynomial) over rings.}
\end{center}

The classical {\em division polynomials} $\Psi_n \in \bb Z[a_1,a_2,a_3,a_4,a_6,X,Y]$ (and the related polynomials $\Omega_n, \Phi_n \in \bb Z[a_1,a_2,a_3,a_4,a_6,X,Y]$) describe multiples of {\em affine points}---these are points of the form $(a:b:1)$---for all smooth Weierstrass curves over fields, in the following way.
\begin{prop}[{\cite[Prop.\ 3.55]{enge}}]
  Let $n$ be a positive integer.
  Then for all fields $k$, all smooth Weierstrass curves over $k$ defined by
  \begin{equation}
    Y^2 + a_1 XY + a_3 Y = X^3 + a_2 X^2 + a_4 X + \alpha_6,
    \label{intro_eqn_weierstrass}
  \end{equation}
  and all $P = (x:y:1)$ on $E$, we have $nP = \pth*[big]{\Phi_n \Psi_n : \Omega_n : \Psi_n^3}(a_1,a_2,a_3,a_4,a_6,x,y,1)$.
\end{prop}
We recall the definition of the division polynomials in \autoref{const_sect_generic}, which is based on \cite[Ch.~3]{enge}.
The result above is not completely satisfactory, as the division polynomials are not defined on non-affine points.

In this paper we construct for all integers $n$ three homogeneous polynomials $\alpha_n, \beta_n, \gamma_n$ in $\bb Z[a_1,a_2,a_3,a_4,a_6][x,y,z]$ (unique up to a common sign) of degree $n^2$, as homogenisations of the unique representatives of the polynomials $\Phi_n \Psi_n, \Omega_n, \Psi_n^3$ modulo the polynomial (\ref{intro_eqn_weierstrass}) having $X$-degree at most $2$.
We show that for all Weierstrass curves $C$ over a ring $R$, and all ``nowhere singular'' points on $C$ of the form $P = (x:y:z)$, we have $nP = \pth*[big]{\alpha_n(P):\beta_n(P):\gamma_n(P)}$ (see the example on the next page for $\alpha_2$, $\beta_2$, $\gamma_2$).
The proof that $\alpha_n$, $\beta_n$, $\gamma_n$ has this property uses the Theorem of the Cube (see e.g.~\cite[Thm.~IV.3.3]{raynaud}).

For an application of this result, consider the modular curve $Y_1(n)$ for $n \geq 4$, i.e.~the moduli space of elliptic curves $E$ together with an {\em embedding} of $\bb Z/n \bb Z$ into $E$.
Explicit equations for its function field have been given before, for example in \cite{baaziz} over $\bb C$, or in \cite{sutherland} over an arbitrary field.
We describe one set of such equations below (see also e.g.~\cite[Sect.~3]{baaziz}).

Define $\psi_n \in \bb Z[1/n,s,t]$ by $\psi_n = \Psi_n(1+s,t,t,0,0,0,0)$.
Then $\psi_d \mid \psi_n$ if $d \mid n$ (see \autoref{moduli_divide}), so we may define $f_n$ recursively by
\[
  f_n = \left\{
    \begin{array}{ll}
      1 & \text{if }n = 1\\
      \psi_n \prod_{d\mid n, 0<d<n} f_d^{-1} & \text{if }n \geq 2
    \end{array}
  \right.,
\]
We show that the curve $Y_1(n)$ (rather than its function field) is defined over $\bb Z[1/n]$ by $f_n = 0$ and $\delta \neq 0$, where $\delta \in \bb Z[s,t]$ is the discriminant of the Weierstrass curve 
\[
  y^2z + (1+s) xyz + tyz^2 = x^3 + tx^2z.
\]

\subsubsection*{Overview}
In \autoref{const_sect_defn}, we define Weierstrass curves over arbitrary schemes, and give some properties of them that we will use throughout this paper.
In \autoref{const_sect_generic}, we describe the construction of $\alpha_n$, $\beta_n$, $\gamma_n$.
In \autoref{const_sect_main}, we state the main theorem.
\autoref{const_sect_generic} and \ref{const_sect_main}, as well as the summary of \autoref{const_sect_defn}, should be accessible for anyone with some basic knowledge of commutative algebra and of elliptic curves over fields.

In \autoref{univ_sect_smooth}, we consider multiplication by $n$ on smooth Weierstrass curves, by reducing to the {\em universal smooth Weierstrass curve} and use the Theorem of the Cube to show the existence of three homogeneous polynomials of degree $n^2$ defining multiplication by $n$.
We compare these polynomials with $\alpha_n$, $\beta_n$, $\gamma_n$, and show that they are the same.
In \autoref{univ_sect_sing}, we show that $\alpha_n$, $\beta_n$, $\gamma_n$ in fact define multiplication by $n$ on all Weierstrass curves, using the fact that they do so on the universal smooth Weierstrass curve.

Finally, in \autoref{moduli_sect_moduli}, we apply the main result for smooth Weierstrass curves, to obtain the aforementioned explicit description of the moduli space of elliptic curves $E$ together with an embedding of $\bb Z/n\bb Z$ into $E$.

\begin{ex}[Doubling formula] \label{intro_doubling}
  Let $R$ be a ring, and let $C$ be the Weierstrass curve over $R$ defined by the polynomial
  \[
    W = y^2z + a_1xyz + a_3yz^2 - x^3 - a_2x^2z - a_4xz^2 - a_6z^3,
  \]
  where $a_1, a_2, a_3, a_4, a_6 \in R$.
  Let $P = (x:y:z)$ be a point on $C$ given in homogeneous coordinates, such that $(\partial W / \partial x)(P)$, $(\partial W / \partial y)(P)$, $(\partial W / \partial z)(P)$ generate the unit ideal in $R$.
  Then
  \[
    2P = \pth*[big]{\alpha_2(P):\beta_2(P):\gamma_2(P)},
  \]
  where
  {\footnotesize
    \begin{align*}
      \alpha_2 &= 2xy^3 + 3a_1x^2y^2 + (a_1^2 - 2a_2)y^3z + (a_1^3 - 3a_1a_2 + 3a_3)xy^2z + (-2a_1^2a_2 + 2a_2^2 - 6a_4)x^2yz \\
      &\qquad \mbox{} + (a_1a_2^2 - 3a_2a_3 - 3a_1a_4)y^2z^2 + (a_1^2a_2^2 - a_1^3a_3 - 2a_1a_2a_3 - 4a_1^2a_4 - 3a_3^2 + 2a_2a_4 - 18a_6)xyz^2 \\
      &\qquad \mbox{} + (-a_1a_2^3 + a_1^2a_2a_3 + a_2^2a_3 - 3a_1a_3^2 + 4a_1a_2a_4 - 3a_3a_4 - 9a_1a_6)x^2z^2 \\
      &\qquad \mbox{} + (a_1a_2^2a_3 - a_1^2a_3^2 - 3a_2a_3^2 - a_1a_3a_4 - 3a_1^2a_6 + 2a_4^2 - 6a_2a_6)yz^3 \\
      &\qquad \mbox{} + (-a_1a_2a_3^2 - a_1a_2^2a_4 + 2a_1^2a_3a_4 - a_1^3a_6 - 2a_3^3 + a_2a_3a_4 + 4a_1a_4^2 - 3a_1a_2a_6 - 9a_3a_6)xz^3 \\
      &\qquad \mbox{} + (-a_2a_3^3 + a_1a_3^2a_4 - a_1a_2^2a_6 + a_3a_4^2 - 3a_2a_3a_6 + 3a_1a_4a_6)z^4, \\
      \beta_2 &= y^4 + a_1xy^3 + (a_1a_2 - 2a_3)y^3z + (a_1^2a_2 - a_2^2 - 3a_1a_3 + 3a_4)xy^2z \\
      &\qquad \mbox{} + (-2a_1a_2^2 + 6a_1a_4)x^2yz + (a_2^3 - a_1a_2a_3 + a_1^2a_4 - 5a_2a_4 + 18a_6)y^2z^2 \\
      &\qquad \mbox{} + (a_1a_2^3 - 2a_1^2a_2a_3 + a_1^3a_4 - a_2^2a_3 + 3a_1a_3^2 - 6a_1a_2a_4 + 3a_3a_4 + 27a_1a_6)xyz^2 \\
      &\qquad \mbox{} + (-a_2^4 + 2a_1a_2^2a_3 - a_1^2a_2a_4 + 6a_2^2a_4 - 6a_1a_3a_4 + 9a_1^2a_6 - 9a_4^2)x^2z^2 \\
      &\qquad \mbox{} + (a_2^3a_3 - a_1a_2a_3^2 + a_1^3a_6 + 2a_3^3 - 5a_2a_3a_4 - a_1a_4^2 + 3a_1a_2a_6 + 18a_3a_6)yz^3 \\ 
      &\qquad \mbox{} + (a_1^2a_2a_3^2 - a_1^3a_3a_4 + a_1^4a_6 + 2a_2^2a_3^2 - a_1a_3^3 - a_2^3a_4 - 2a_1^2a_4^2 + 6a_1^2a_2a_6 - 6a_3^2a_4 + 3a_2a_4^2 + 9a_2^2a_6 - 27a_4a_6)xz^3 \\
      &\qquad \mbox{} + (a_1a_2a_3^3 - a_1^2a_3^2a_4 + a_1^3a_3a_6 - a_3^4 + a_2a_3^2a_4 - 2a_1a_3a_4^2 - a_2^3a_6 + 6a_1a_2a_3a_6 - a_4^3 - 9a_3^2a_6 + 9a_2a_4a_6 - 27a_6^2)z^4, \\
      \gamma_2 &= 8y^3z + 12a_1xy^2z + 6a_1^2x^2yz + (a_1^3 + 12a_3)y^2z^2 + (a_1^4 + 12a_1a_3)xyz^2 + (-a_1^3a_2 + 3a_1^2a_3)x^2z^2 \\
      &\qquad \mbox{} + (a_1^3a_3 + 6a_3^2)yz^3 + (-a_1^3a_4 + 3a_1a_3^2)xz^3 + (-a_1^3a_6 + a_3^3)z^4.
    \end{align*}
  }
\end{ex}

\begin{ex}
  Let $C$ be the Weierstrass curve over $\bb Q_2$ defined by the polynomial $W = y^2z - x^3 - z^3$.
  Then note that $\alpha_2 = 2xy^3 - 18xyz^2$, $\beta_2 = y^4 + 18y^2z^2 - 27z^4$, and $\gamma_2 = 8y^3z$.
  Let $P$ be a point of $C$ reducing to $(2:1:8)$ modulo $16$, which exists by Hensel's Lemma.
  Note that $(\partial W/\partial z)(P) \equiv 1 \pmod{16}$, so we can apply \autoref{intro_doubling}.
  Hence $2P$ reduces to $(4:1:0)$ modulo $16$, $4P$ reduces to $(8:1:0)$ modulo $16$, and $8P$ reduces to $(0:1:0)$ modulo $16$.
\end{ex}


\section{Definitions, statement of main theorem}

Rings in this paper are always assumed to be commutative, associative, and unital.

\subsection{Weierstrass curves over schemes}
\label{const_sect_defn}

For now, we will use the language of schemes.
For a summary of this section in terms of Weierstrass curves over rings, we refer to the end of this section.


We start by giving some general definitions.
\begin{defn}
  Let $S$ be a scheme.
  A {\em genus $1$ curve} over $S$ is a flat, proper $S$-scheme locally of finite presentation, of which the geometric fibres are integral curves of arithmetic genus $1$.
  A {\em pointed genus $1$ curve} over $S$ is a pair $(C,s)$ of a genus $1$ curve over $S$, and a section $s \in C(S)$.
  An {\em elliptic curve} over $S$ is a pointed genus $1$ curve $(C,s)$ in which $C$ is smooth over $S$.
\end{defn}

We now define the main objects of study in this paper.
\begin{defn}
  A {\em Weierstrass curve} $C$ over $S$ is a closed subscheme of $\bb P^2_S$ defined by a homogeneous polynomial of the form
  \begin{equation}
    W = y^2z + a_1xyz + a_3yz^2 - x^3 - a_2x^2z - a_4xz^2 - a_6z^3
    \label{const_eqn_weierstrass}
  \end{equation}
  with $a_1, a_2, a_3, a_4, a_6 \in \sh O_S(S)$.
  The {\em discriminant} $\Delta \in \sh O_S(S)$ of $C$ is the discriminant of the equation (\ref{const_eqn_weierstrass}), see e.g.~\cite[Def.~2.7]{enge}.
\end{defn}
We will show later in \autoref{const_flatlfp} that a Weierstrass curve over a scheme $S$ is in fact a genus $1$ curve over $S$.

We now describe the functor of points of the Weierstrass curve $C$ explicitly.
If $T$ is an $S$-scheme, $\sh L$ is an invertible $\sh O_T$-module, $s_0, s_1, s_2 \in \sh L(T)$, and $f \in \sh O_S(S)[x,y,z]$ homogeneous of degree $d$, then we have a well-defined section $f(s_0,s_1,s_2) \in \sh L^{\otimes d}(T)$.
We say that $s_0, s_1, s_2$ {\em generate} $\sh L$ if $(s_0)_x, (s_1)_x, (s_2)_x$ generate $\sh L_x$ as an $\sh O_{T,x}$-module for all $x \in T$.
The following is a special case of a well-known result, see e.g.~\cite[Lem.~\href{http://stacks.math.columbia.edu/tag/01O4}{01O4}]{stacksproject}, as the graded $\sh O_S$-algebra $\sh O_S[x,y,z]/(W)$ is generated by its degree $1$ part.

\begin{prop}\label{const_functorofpoints}
  Let $S$ be a scheme, and let $C$ be a Weierstrass curve over $S$ defined by $W$ as in (\ref{const_eqn_weierstrass}).
  Then for all $S$-schemes $T$, we have
  \[
    C(T) = \set{(\sh L,s_0,s_1,s_2) : 
      \begin{array}{c}
        \sh L \text{ invertible $\sh O_T$-module},\ s_0,s_1,s_2 \in \sh L(T) \text{ generating $\sh L$} \\
        W(s_0,s_1,s_2) = 0 \text{ in $\sh L^{\otimes 3}$}
      \end{array}
    }\Big/\iso.
  \]
  Here, $(\sh L, s_0, s_1, s_2) \iso (\sh M, t_0, t_1, t_2)$ if and only if there exists an isomorphism $\sh L \to \sh M$ of $\sh O_T$-modules mapping $s_i$ to $t_i$ for all $i \in \set{0,1,2}$.
\end{prop}

If $\sh L = \sh O_T$, then we denote the class of $(\sh L, s_0, s_1, s_2)$ by $(s_0:s_1:s_2)$.

\begin{prop}\label{const_flatlfp}
  Let $S$ be a scheme, and let $C$ be a Weierstrass curve over $S$ defined by $W$ as in (\ref{const_eqn_weierstrass}).
  Then $\pth*[big]{C,(0:1:0)}$ is a pointed genus $1$ curve over $S$.
\end{prop}

\begin{proof}
  Except for flatness, our claim is well-known.
  We check affine locally on the base that $C$ is flat over $S$.
  So assume $S = \Spec R$.
  Then $C$ is covered by the two standard affine open covers in which $y$ and $z$, respectively, are invertible.
  Since $W$ is monic with respect to $x$, the corresponding $R$-algebras are free as $R$-modules, hence flat.
  We deduce that $C$ is flat over $S$.
\end{proof}

We deduce that $C$ is smooth (of relative dimension $1$) if and only if $C$ is smooth on all geometric fibres, i.e.~$\Delta \neq 0$ in all geometric points of $S$ (see e.g.~\cite[Prop.~2.25]{enge}).
This holds if and only if $\Delta \in \sh O_S(S) \mg$.

In general, let $C\sm$ be the smooth locus of $C$ over $S$.
Then we have the following description of the functor of points of $C\sm$.

\begin{prop}\label{const_functorofsmoothpoints}
  Let $S$ be a scheme, and let $C$ be a Weierstrass curve over $S$ defined by $W$ as in (\ref{const_eqn_weierstrass}).
  Then for all $S$-schemes $T$, the set $C\sm(T)$ is the subset of $C(T)$ consisting of the classes of the $4$-tuples $(\sh L, s_0, s_1, s_2)$ such that the sections
  \[
    (\partial W / \partial x)(s_0,s_1,s_2),(\partial W / \partial y)(s_0,s_1,s_2),(\partial W / \partial x)(s_0,s_1,s_2) \in \sh L^{\ts 2}(T)
  \]
  generate $\sh L^{\ts 2}$.
\end{prop}

\begin{proof}
  We denote the structure morphism of $C$ by $f$.
  Note that for all $x \in C$, we have that $C$ is smooth in $x$ if and only if $C_{f(x)}$ is smooth in $x$, using \autoref{const_flatlfp}.
  The polynomials $\partial W/\partial x, \partial W/\partial y, \partial W/\partial z$ define the singular locus on every fibre, so $C\sm$ is the open subscheme of $C$ that is the complement of the common zero locus of these polynomials.
\end{proof}

The following shows by \autoref{const_flatlfp} that a smooth Weierstrass curve over a scheme $S$ admits a unique group scheme structure with zero section $(0:1:0)$, which is commutative.
\begin{thm}[{\cite[Prop.~II.2.7]{delignerapoport}}]\label{const_groupscheme}
  Let $S$ be a scheme, and let $C$ be a smooth genus $1$ curve over $S$.
  Then there exists a unique group scheme structure on $C$ with zero section $(0:1:0) \in C(S)$, which is commutative.
\end{thm}

Now we will show that any Weierstrass curve over a scheme $S$ admits a natural group scheme structure on its smooth locus.
So let $C$ be an arbitrary Weierstrass curve over a scheme $S$.
Note that $(0:1:0) \in C\sm(S)$, since $(\partial W / \partial z)(0,1,0) = 1$ generates $\sh O_S$.
We first consider the {\em universal Weierstrass curve}.

\begin{defn}
  Let $\uni A = \bb Z[a_1,a_2,a_3,a_4,a_6]$. Then the {\em universal Weierstrass curve} $\uni C$ is the Weierstrass curve over $\uni A$ defined by
  \[
    \uni W = y^2z + a_1xyz + a_3yz^2 - x^3 - a_2x^2z - a_4xz^2 - a_6z^3 \in \uni A[x,y,z].
  \]
\end{defn}

It is universal in the following sense.

\begin{prop}\label{const_univ}
  Let $S$ be a scheme, and let $C$ be a Weierstrass curve over $S$ defined by $W$ as in (\ref{const_eqn_weierstrass}).
  Then there exist unique morphisms $g \colon C \to \uni C$ and $h \colon S \to \Spec \uni A$ such the following diagram is commutative, and such that the outer square is Cartesian.

  \[
    \begin{tikzcd}
      C \arrow{d}{g} \arrow{r} & \bb P^2_S \arrow{d}{h_{\bb P^2_{\uni A}}} \arrow{r} & S \arrow{d}{h} \\
      \uni C \arrow{r} & \bb P^2_{\uni A} \arrow{r} & \Spec \uni A
    \end{tikzcd}
  \]
\end{prop}

\begin{proof}
  It suffices to show this in the case that $S$ is affine; the general case follows from it by gluing.

  So suppose that $S = \Spec A$, and let $\phi$ denote the morphism $S \to \Spec \uni A$ given by the ring morphism $\uni A \to A$ sending every $a_i$ to the corresponding coefficient of $W$.
  Then $C = \uni C \fp[\uni A] S$ (via $\phi$), so $(\phi_{\uni C},\phi)$ is a pair satisfying the required properties.

  Suppose that $(g,h)$ is a pair of morphisms satisfying the required properties.
  Then note that $\uni C \to \bb P^2_{\uni A}$ is a closed immersion, hence a monomorphism, so it follows that $g$ is the unique morphism making the diagram commute.
  Moreover, $h_{\bb P^2_{\uni A}}$ induces a map from $C$ to $\uni C$ if and only if the induced map $\sh O_{\bb P^2_{\uni A}}(3)(\bb P^2_{\uni A}) \to \sh O_{\bb P^2_A}(3)(\bb P^2_A)$ (which is the map $\uni A[x,y,z]_3 \to A[x,y,z]_3$ induced by the ring morphism $\uni A \to A$) sends $\uni W$ to $W$, in other words, if and only if $h = \phi$.
\end{proof}

We now have the following corollary of \autoref{const_groupscheme}.

\begin{cor}\label{const_groupscheme_cor}
  Let $\uni S = \Spec \uni A$, and let $\uni C$ be the universal Weierstrass curve over $\uni S$.
  There exists a unique commutative group scheme structure on $\uni C\sm$ that has zero section $(0:1:0) \in \uni C\sm(\uni S)$.
\end{cor}

\begin{proof}
  First note that by \autoref{const_flatlfp} and \autoref{const_groupscheme}, $\uni C\sm$ admits a group scheme structure with zero section $(0:1:0) \in \uni C\sm(\uni A)$.
  We show that it must be unique.

  Note that $\uni C\sm \fp[\uni A] \uni C\sm$ is smooth over $\uni A$, with irreducible fibres, and that $\uni A$ is integral.
  Hence (for example by combining \cite[Lem.~\href{http://stacks.math.columbia.edu/tag/004Z}{004Z}]{stacksproject} and \cite[Prop.~17.5.7]{ega43}) $\uni C\sm \fp[\uni A] \uni C\sm$ is an integral scheme.
  Also note that $\uni C\sm$ is separated over $\uni A$.

  Now consider $\uni R = \bb Z[a_1,a_2,a_3,a_4,a_6,1/\Delta]$.
  Then $\Spec R$ is an open subscheme of $\Spec A$.
  Let $\uni E = \uni C \fp[\uni A] \Spec \uni R$, and note that it is a non-empty open subscheme of $\uni C\sm$.
  As $\uni E$ is a smooth Weierstrass curve, $\uni E$ admits a unique group scheme structure with zero section $(0:1:0)$ in $\uni E(\uni R)$.
  Hence the group scheme structure on $\uni C\sm$ extends that of $\uni E$, so by the above this must be the unique extension of the group scheme structure of $\uni E$ to $\uni C\sm$.
\end{proof}

By universality of $\uni C$, this gives a natural commutative group scheme structure on every Weierstrass curve.

We now summarise in terms of Weierstrass curves over rings.
\begin{summ}
  Let $R$ be a ring.
  A {\em Weierstrass curve} $C$ over $R$ is a closed subscheme of $\bb P^2_R$ defined by a homogeneous polynomial of the form (\ref{const_eqn_weierstrass}) with $a_1,a_2,a_3,a_4,a_6 \in R$.
  The {\em discriminant} $\Delta \in R$ of $C$ is the discriminant of (\ref{const_eqn_weierstrass}), see e.g.~\cite[Def.~2.7]{enge}.
  The Weierstrass curve $C$ is smooth if and only if $\Delta \in R\mg$.

  We have an explicit description of the set $C(R)$ of $R$-valued points of $C$ as follows.
  If $M$ is an invertible $R$-module, $m_0, m_1, m_2 \in M$, and $f \in R[x,y,z]$ homogeneous of degree $d$, then we have $f(s_0,s_1,s_2) \in M^{\otimes d}$.
  Then
  \[
    C(R) = \set{(M,m_0,m_1,m_2) : 
      \begin{array}{c}
        M \text{ invertible $R$-module},\ m_0,m_1,m_2 \in M \\
        Rm_0+Rm_1+Rm_2=M,\ W(m_0,m_1,m_2) = 0 \text{ in $M^{\otimes 3}$}
      \end{array}
    }\Big/\iso,
  \]
  where $(M,m_0,m_1,m_2) \iso (N,n_0,n_1,n_2)$ if and only if there exists an isomorphism $M \to N$ of \mbox{$R$-modules} mapping $m_i$ to $n_i$ for all $i \in \set{1,2,3}$.
  If $M = R$, the class of $(M,m_0,m_1,m_2)$ is denoted $(m_0:m_1:m_2)$.

  As an example, if $R$ is a field (or more generally, any ring with trivial Picard group), then
  \[
    C(R) = \set*{(x,y,z)\in R^3: W(x,y,z)=0,\ Rx+Ry+Rz=R}/R\mg.
  \]

  In general, let $C\sm$ be the open subscheme of $C$ that is the complement of the common zero locus of $\partial W / \partial x, \partial W / \partial y, \partial W / \partial x$. 
  Then $C\sm(R)$ is the subset of $C(R)$ consisting of the classes of the $4$-tuples $(M, m_0, m_1, m_2)$ such that 
  \[
    (\partial W / \partial x)(m_0,m_1,m_2),\ (\partial W / \partial y)(m_0,m_1,m_2),\ (\partial W / \partial z)(m_0,m_1,m_2)
  \]
  generates $M^{\otimes 2}$.

  The set $C\sm(R)$ contains the point $(0:1:0)\in C(R)$, and has a natural structure of an abelian group with $(0:1:0)$ as neutral element.
\end{summ}

\subsection{Division polynomials and construction of $\alpha_n$, $\beta_n$, $\gamma_n$}
\label{const_sect_generic}

Let us recall the definition of division polynomials.
The main reference for this is \cite[Ch.~3]{enge}.
For this purpose, we consider a special (but yet also generic!) smooth Weierstrass curve.

Let $K$ be the algebraic closure of the field of fractions of the ring $\uni A = \bb Z[a_1,a_2,a_3,a_4,a_6]$.
Let $E$ be the elliptic curve over $K$ defined by the homogeneous Weierstrass polynomial
\[
  \uni W = y^2z + a_1xyz + a_3yz^2 - x^3 - a_2xz^2 - a_4x^2z - a_6z^3.
\]
Further, let $X$, $Y$ denote the rational functions $x/z$ and $y/z$, respectively.
Then the function field $K(E)$ of $E$ is the field of fractions of $K[X,Y]/(\uni W')$, where 
\[
  \uni W' = Y^2 + a_1XY + a_3Y - X^3 - a_2X^2 - a_4X - a_6
\]
is the affine Weierstrass polynomial.
We have a leading coefficient map $\Lambda \colon K(E)-0 \to K-0$, given by $f \mapsto \pth*[big]{(X/Y)^{-\ord[0]f} f}(0)$, where $\ord[0] f$ denotes the order of $f$ at the point at infinity of $E$.
This is well-defined as $X/Y$ has a simple zero at $0$.

Then we can define the division polynomials as follows.
\begin{defn}
  Let $n \in \bb Z - \set 0$.
  The {\em $n$-th division polynomial} $\Psi_n$ is the unique rational function $f$ in $K(E)$ with divisor $\sum_{P \in E[n]-0} \div{P} - (n^2-1) \div{0}$ and leading coefficient $n$.
  Moreover, we define $\Psi_0$ to be $0$.
\end{defn}

These division polynomials exist, as the degree of the divisors defining them is $0$, and for all $n \in \bb Z - 0$, we have $\sum_{P \in E[n]} P = 0$.

\begin{rk}
  The division polynomials satisfy the following relation for all $m, n \in \bb Z$ (see \cite[Prop.~3.53]{enge})
  \[
    \Psi_{m+n}\Psi_{m-n} = \Psi_{m+1}\Psi_{m-1}\Psi_n^2 - \Psi_{n+1}\Psi_{n-1}\Psi_m^2.
  \]
  Using $\Psi_1$, $\Psi_2$, $\Psi_3$, and $\Psi_4$, and the recurrence relation above, one can recursively compute the other division polynomials.
\end{rk}

Furthermore, define the following rational functions.
\begin{align*}
  \Phi_n &= X\Psi_n^2 - \Psi_{n-1}\Psi_{n+1} &
  \Omega_n &= \left\{
    \begin{array}{ll}
      1 & \text{if $n = 0$}\\
      \tfrac1{2\Psi_n}\pth*[big]{\Psi_{2n} - \Psi_n^2(a_1 \Phi_n + a_3 \Psi_n^2)} & \text{otherwise}
    \end{array}
    \right.
  \end{align*}

  We list some facts in the next propositions.
  \begin{prop}[{\cite[Sect.~3.6]{enge}}]
    Let $n \in \bb Z$.
    Then $\Psi_n, \Phi_n, \Omega_n \in \uni A[X,Y]/(\uni W')$.
  \end{prop}

  \begin{prop}
    Let $n \in \bb Z$.
    Then $\Lambda \Phi_n = \Lambda \Omega_n = 1$, and $\ord[0] \Phi_n = -2n^2$, $\ord[0] \Omega_n = -3n^2$.
  \end{prop}

  \begin{proof}
    This is a straightforward calculation, using that $\Lambda \Psi_n = n$ and $\ord[0] \Psi_n = -(n^2 - 1)$.
  \end{proof}

  Now we can state the basic result mentioned in the introduction more precisely, in the case of our special smooth Weierstrass curve $E/K$.
  \begin{prop}[{\cite[Prop.\ 3.55]{enge}}]\label{const_basic}
    Let $P = (x:y:1) \in E(K)$ be a rational point, and let $n \in \bb Z$.
    Then 
    \begin{align*}
      nP &= \pth*[big]{\Psi_n(x,y)\Phi_n(x,y):\Omega_n(x,y):\Psi_n^3(x,y)}.
    \end{align*}
  \end{prop}

  We want to obtain a version of this result that works for all rational points, including the point at infinity.
  We will do this by choosing representatives in $\uni A[X,Y]$ of the elements $\Psi_n\Phi_n$, $\Omega_n$, $\Psi_n^3$ of $\uni A[X,Y]/(\uni W')$ and then homogenising them.

  However, one has to be careful with the choice of representatives.
  As an example of this, the usual representatives chosen for the $\Psi_n$ have $Y$-degree at most $1$, which is convenient for calculations.
  But this also causes homogenisations to be divisible by high powers of $z$ for large $n$, as the leading term with respect to $y$ in the homogeneous Weierstrass equation $\uni W$ is $y^2z$.
  This causes problems at the point at infinity for $\abs n \geq 2$.

  Note that $\uni W$ is monic with respect to $x$, which motivates the following definition.
  Recall that $X = x/z$ and $Y = y/z$.

  \begin{defn}\label{const_newdivpol}
    Let $A_n, B_n, C_n \in \uni A[X,Y]$ be the unique representatives of $\Psi_n\Phi_n$, $\Omega_n$, $\Psi_n^3$ with $X$-degree at most $2$, and define the polynomials $\alpha_n = A_n z^{n^2}$, $\beta_n = B_n z^{n^2}$, $\gamma_n = C_n z^{n^2}$ in $\uni A[x,y,z,z^{-1}]$.
  \end{defn}
  Uniqueness of $A_n$, $B_n$, $C_n$ is guaranteed by the fact that $\uni W'$ is monic in $X$.

  \begin{prop}\label{const_integral}
    Let $n \in \bb Z$.
    Then $\alpha_n, \beta_n, \gamma_n \in \uni A[x,y,z]$ and are homogeneous of degree $n^2$.
    Moreover, $\alpha_n, \gamma_n \in (x,z)$ and $\beta_n \in y^{n^2} + (x,z)$.
  \end{prop}

  \begin{proof}
    To show that $\alpha_n, \beta_n, \gamma_n \in \uni A[x,y,z]$, it suffices to show that the polynomials $A_n$, $B_n$, $C_n$ have total degree at most $n^2$.
    Note that the monomials with $X$-degree at most $2$ all have distinct orders at infinity.
    As $A_n$, $B_n$, $C_n$ have orders $-3n^2 + 1$, $-3n^2$, $-3n^2 + 3$ at infinity, respectively, and as they have $X$-degree at most $2$, it follows that they have total degree at most $n^2$.
    (Here, we use that $X$ and $Y$ have poles of orders $2$ and $3$ at the point at infinity, respectively.)

    Now $\alpha_n, \beta_n, \gamma_n$ are by construction homogeneous of degree $n^2$, and as $Y^{n^2}$ is the unique monomial with order $-3n^2$, the result follows.
  \end{proof}

  For our special smooth Weierstrass curve $E$, the problem at the point at infinity is now solved.

  \begin{prop}\label{const_main_special}
    Let $n \in \bb Z$, and let $P = (x:y:z) \in E(K)$.
    Then
    \[
      nP = \pth*[big]{\alpha_n(P):\beta_n(P):\gamma_n(P)}.
    \]
  \end{prop}

  \begin{proof}
    As $\alpha_n$, $\beta_n$, $\gamma_n$ were obtained from $\Psi_n\Phi_n$, $\Omega_n$, $\Psi_n^3$ by homogenisation, the result follows from \autoref{const_basic} for $P$ of the form $(x:y:1)$.
    For $P = (0:1:0)$, the claim follows from \autoref{const_integral}.
  \end{proof}

  \subsection{Statement of the main theorem} \label{const_sect_main}

  Let $R$ be a ring.
  Given $W$ in $R[x,y,z]$ of the form (\ref{const_eqn_weierstrass}), there is an obvious ring morphism $\uni A \to R$ which maps $a_i \in \uni A$ to the corresponding coefficient of $W$ in $R$.
  We will view $R$ as an $\uni A$-algebra via this morphism.

  \begin{thm}\label{const_main}
    Let $R$ be a ring, and let $C$ be a Weierstrass curve over $R$ defined by $W$ as in (\ref{const_eqn_weierstrass}).
    Let $P = [(M,m_0,m_1,m_2)] \in C\sm(R)$. Then
    \[
      nP = \brc*[Big]{\pth*[big]{M^{\otimes n^2}, \alpha_n(m_0,m_1,m_2), \beta_n(m_0,m_1,m_2), \gamma_n(m_0,m_1,m_2)}}
    \]
  \end{thm}
  In Section 2, we will state and prove the scheme-theoretic equivalent of this theorem.

  If $M = R$, then $M^{\otimes n^2} = R$, so if a point $P$ is of the form $(x:y:z)$, then $nP$ is also of this form, namely,
  \[
    nP = \pth*[big]{\alpha_n(x,y,z):\beta_n(x,y,z):\gamma_n(x,y,z)}
  \]
  However, even for a smooth Weierstrass curve $C$, the subset of $R$-valued points of the form $(x:y:z)$ need not be closed under addition, as we will see in \autoref{const_closed}.

  As $\alpha_n$, $\beta_n$, $\gamma_n$ were obtained from $\Phi_n\Psi_n$, $\Omega_n$, $\Psi_n^3$ by homogenisation, we then have the following immediate consequence of \autoref{const_main}.

  \begin{cor}\label{const_main_cor}
    Let $n \in \bb Z$, and let $C$ be a Weierstrass curve over a ring $R$.
    Then for all $P \in C\sm(R)$ of the form $(x:y:1)$, we have
    \[
      nP = \pth*[big]{\Phi_n(x,y)\Psi_n(x,y):\Omega_n(x,y):\Psi^3_n(x,y)}.
    \]
  \end{cor}

  \begin{ex}\label{const_closed}
    We give an example of a Weierstrass curve $C$ over a ring $R$ and two $R$-valued points $P, Q \in C(R)$ of the form $(x:y:z)$ such that its sum $P+Q \in C(R)$ is not of the form $(x:y:z)$.
    Let $R = \bb Z\brc*[big]{\sqrt{-5}}$, let $K = \bb Q\pth*[big]{\sqrt{-5}}$, and consider the smooth Weierstrass curve $C$ over $R$ given by 
    \[
      W = y^2z + xyz + yz^2 - x^3 - 4xz^2 + 6z^3,
    \]
    and the two points
    \begin{align*}
      P &= (9 : 23 : 1) & Q &= \bigl(3411\sqrt{-5} : 26488 + 117\sqrt{-5} : -3645\sqrt{-5}\bigr)
    \end{align*}
    in $C(K)$.
    Note that $P$ and $Q$ define points of the form $(x:y:z)$ in $C(R)$ as well, as the $R$-submodule of $K$ generated by their coordinates are trivial, but that the $R$-submodule of $K$ generated by the coordinates of
    \[
      P + Q = \bigl(61028487 + 104922279\sqrt{-5} : 120011054 - 171672039\sqrt{-5} : -127263527\bigr) \in C(K)
    \]
    is invertible, but not principal, so $P+Q$ is not of the form $(x:y:z)$ in $C(R)$.
  \end{ex}


\section{Proof of the main theorem}

Here is the scheme-theoretic equivalent of \autoref{const_main}.

\begin{thm}\label{univ_main}
  Let $S$ be a scheme, and let $C$ be a Weierstrass curve over $S$ defined by $W$ as in (\ref{const_eqn_weierstrass}).
  Let $T$ be an $S$-scheme, and let $P = [(\sh L,s_0,s_1,s_2)] \in C\sm(T)$.
  Then
  \[
    nP = \brc*[Big]{\pth*[big]{\sh L^{\otimes n^2}, \alpha_n(s_0,s_1,s_2), \beta_n(s_0,s_1,s_2), \gamma_n(s_0,s_1,s_2)}}.
  \]
\end{thm}

We will first prove this for smooth Weierstrass curves, and we will do this by first using the Theorem of the Cube, by considering a {\em universal point} on a {\em universal smooth Weierstrass curve}.

\subsection{The universal smooth Weierstrass curve and the proof in the smooth case} \label{univ_sect_smooth}

\begin{defn}
  The {\em universal smooth Weierstrass curve} $\uni E$ is the Weierstrass curve over $\Spec \uni R$ (where $\uni R = \bb Z[a_1,a_2,a_3,a_4,a_6,1/\Delta]$) defined by the homogeneous polynomial
  \[
    \uni W = y^2z + a_1xyz + a_3yz^2 - x^3 - a_2x^2z - a_4xz^2 - a_6z^3 \in \uni R[x,y,z].
  \]
  The {\em universal point} $\uni P$ on $\uni E$ is the point in $\uni E(\uni E)$ corresponding to the identity map on $\uni E$.
\end{defn}

Note here that $\uni P = \brc*[big]{\pth*[big]{\sh O_{\uni E}(1),x,y,z}}$, and that it is universal in the sense that any morphism $S \to \uni E$ of $\uni R$-schemes factors through the identity on $\uni E$.
Moreover, $\uni E$ is universal in the following sense.

\begin{prop} \label{univ_univ}
  Let $E$ be a smooth Weierstrass curve over $S$. Then there exist unique morphisms $g \colon E \to \uni E$ and $h \colon S \to \Spec R$ such the following diagram is commutative, and such that the outer square is Cartesian.
  \[
    \begin{tikzcd}
      E \arrow{d}{g} \arrow{r} & \bb P^2_S \arrow{d}{h_{\bb P^2_{\uni R}}} \arrow{r} & S \arrow{d}{h} \\
      \uni E \arrow{r} & \bb P^2_{\uni R} \arrow{r} & \Spec \uni R
    \end{tikzcd}
  \]
\end{prop}

\begin{proof}
  Mimic the proof of \autoref{const_univ}.
\end{proof}

We want to consider the multiples of $\uni P$.
Note that for $n \in \bb Z$, the point $n \uni P$ is the point corresponding to the multiplication-by-$n$ map $\mu_n$ on $\uni E$.
The following is a special case of the Theorem of the Cube.

\begin{thm}[{\cite[Thm.\ IV.3.3]{raynaud}}]\label{univ_cube}
  Let $n_1, n_2, n_3 \in \bb Z$.
  Then there is a canonical isomorphism of $\sh O_{\uni E}$-modules
  \[
    \bigotimes_{I \subs \set{1,2,3}} \mu_{\sum_{i \in I} n_i}^* \sh O_{\uni E}(1)^{(-1)^{\#I}} = \sh O_{\uni E}.
  \]
\end{thm}

\begin{prop}\label{univ_multn}
  Let $n \in \bb Z$.
  Then $\mu_n^* \sh O_{\uni E}(1) \iso \sh O_{\uni E}(n^2)$.
\end{prop}

\begin{proof}
  As $\uni E$ is a smooth Weierstrass curve, we have $\sh O_{\uni E}(1) = I^{-3}(0)$;
  the identity on $\uni E$ can be given by $\brc*[big]{\pth*[big]{\sh O_{\uni E}(1),x,y,z}}$, as well as $\brc*[big]{\pth*[big]{I^{-3}(0),X,Y,1}}$, where the latter is in the notation of~\cite[Sect.~2.2]{katzmazur}.
  As $\mu_{-1}$ fixes the zero section, we deduce that $\mu_{-1}^* \sh O_{\uni E}(1) \iso \sh O_{\uni E}(1)$.
  Hence it suffices to show that for all positive $n \in \bb Z$, we have $\mu_n^* \sh O_{\uni E}(1) \iso \sh O_{\uni E}(n^2)$.

  For $n=2$, we apply \autoref{univ_cube} with $n_1 = n_2 = 1$, $n_3 = -1$ to see that $\mu_2^* \sh O_{\uni E}(1) \iso \sh O_{\uni E}(4)$.
  For $n>2$, our claim follows by induction, by applying \autoref{univ_cube} with $n_1 = n-2$, and $n_2 = n_3 = 1$.
\end{proof}

As a consequence, we have the following important observation.

\begin{cor}\label{univ_multn_cor}
  Let $n \in \bb Z$.
  Then there exist homogeneous elements $\alpha'_n$, $\beta'_n$, $\gamma'_n$ in $\uni R[x,y,z]/(\uni W)$ of degree $n^2$ such that 
  \[
    \brc*[big]{\pth*[big]{\sh O_{\uni E}(n^2),\alpha'_n,\beta'_n,\gamma'_n}} \in \uni E(\uni E)
  \]
  is the point corresponding to $\mu_n$.
  These elements are unique up to a common unit in $\uni R$.
\end{cor}

\begin{proof}
  We simply note that the global sections of $\sh O_{\uni E}(n^2)$ are the homogeneous elements of degree $n^2$ in $\uni R[x,y,z]/(\uni W)$, and that $\Aut_{\sh O_{\uni E}}\pth*[big]{\sh O_{\uni E}(n^2)} = \sh O_{\uni E}(\uni E) \mg = \uni R \mg$.
\end{proof}

Translating this observation in terms of the functor of points of $\uni E$ using the universality of $\uni P$ then gives the following.

\begin{cor}
  Let $n \in \bb Z$.
  Let $T$ be a scheme over $\uni R$, and let $P = [(\sh L, s_0, s_1, s_2)] \in \uni E(T)$.
  Then we have
  \[
    nP = \brc*[big]{\pth*[big]{\sh L^{\ts n^2}, \alpha'_n(s_0,s_1,s_2), \beta'_n(s_0,s_1,s_2), \gamma'_n(s_0,s_1,s_2)}}.
  \]
\end{cor}

Finally, using the universality of $\uni E$, we obtain the following.

\begin{cor}\label{univ_exist}
  Let $n \in \bb Z$.
  Let $S$ be a scheme, and let $E$ be a smooth Weierstrass curve over $S$ defined by $W$ as in (\ref{const_eqn_weierstrass}).
  Let $T$ be an $S$-scheme, and let $P = [(\sh L,s_0,s_1,s_2)] \in E(T)$.
  Then we have
  \[
    nP = \brc*[big]{\pth*[big]{\sh L^{\ts n^2}, \alpha'_n(s_0,s_1,s_2), \beta'_n(s_0,s_1,s_2), \gamma'_n(s_0,s_1,s_2)}}.
  \]
\end{cor}

We now finish the proof of \autoref{univ_main} in the smooth case.

\begin{proof}[Proof of \autoref{univ_main} (smooth case)]
  We will now show that up to a common unit in $\uni R$, the triples $(\alpha_n,\beta_n,\gamma_n)$ and $(\alpha'_n,\beta'_n,\gamma'_n)$ are equal.
  First observe that we have $\beta_0, \beta'_0 \in \uni R \mg$, $\alpha'_0 = 0 = \alpha_0$, and $\gamma'_0 = 0 = \gamma_0$.
  Hence we may assume that $n \neq 0$.

  Consider the smooth Weierstrass curve $E$ over $K$ as defined in \autoref{const_sect_generic}.
  Note that it now suffices to show that $\alpha_n / \alpha'_n = \beta_n / \beta'_n = \gamma_n / \gamma'_n$ as rational functions in $K(E)$, and that this is a unit in $\uni R$, i.e. of the form $\pm \Delta^i$ for some integer $i$.

  Observe that for all $P = (x:y:1) \in E(K)-0$, we have, by \autoref{const_basic} and \autoref{univ_exist},
  \[
    \pth*[big]{\alpha'_n(P):\beta'_n(P):\gamma'_n(P)} = nP = \pth*[big]{\Phi_n(P)\Psi_n(P):\Omega_n(P):\Psi_n^3(P)}
  \]
  First note that by definition, $\alpha_n / \gamma_n = \Phi_n / \Psi_n^2$ and $\beta_n / \gamma_n = \Omega_n / \Psi_n^3$.
  Hence $\alpha_n / \gamma_n = \alpha'_n / \gamma'_n$ and $\beta_n / \gamma_n = \beta'_n / \gamma'_n$, as rational functions in $K(E)$.
  As $\mu_n$ is surjective, it follows that $\im \mu_n$ is infinite, hence none of $\alpha'_n$, $\beta'_n$, $\gamma'_n$ are identically zero.
  So $\alpha_n / \alpha'_n = \beta_n / \beta'_n = \gamma_n / \gamma'_n$ in $K(E)$.
  (These are actually rational functions, as both the numerators and denominators are homogeneous of degree $n^2$.)

  Let $\theta_n = \beta_n / \beta'_n$.
  Then we need to prove that $\theta_n$ is in fact a unit in $\uni R$.

  First note that for all $P \in E(K)$, we have $\pth*[big]{\alpha'_n(P),\beta'_n(P),\gamma'_n(P)} \neq (0,0,0)$.
  Hence $\theta_n$ has no poles on $E$.
  It follows that $\theta_n$ has no zeroes on $E$ either, hence $\theta_n$ is constant, i.e. $\theta_n \in K \mg$.

  As the homogeneous Weierstrass polynomial $\uni W$ is monic in $x$, it follows that $\uni R[x,y,z]/(\uni W)_{n^2}$ is a free $\uni R$-module, with a basis consisting of monomials.
  Since $\theta_n$ is by definition a quotient of two elements of $\uni R[x,y,z]/(\uni W)_{n^2}$, which is moreover a constant, it follows that $\theta_n$ is in the field of fractions of $\uni R$.

  Now observe that $\uni R$ is a unique factorisation domain.
  Hence we can write $\theta_n = f/g$ with $f,g \in \uni R$ in such a way that $f$ and $g$ have no common factors.
  Then $f\beta'_n = g\beta_n$ in $\uni R[x,y,z]/(\uni W)$.
  Thus $g$ divides all coefficients of $\beta'_n$.
  Now note that for the point $P = (0:1:0)$, we have $\pth*[big]{\alpha'_n(P):\beta'_n(P):\gamma'_n(P)} = (0:1:0)$, hence the $y^{n^2}$-coefficient of $\beta'_n$ is a unit.
  Therefore, $g$ is a unit as well, which implies that $\theta_n \in \uni R$.

  Finally, note that $\theta_n$ divides all coefficients of $\beta_n$.
  But because $\Lambda \Omega_n = 1$ and $\ord[0] \Omega_n = -3n^2$, the coefficient of $y^{n^2}$ in $\beta_n$ must also be $1$, as $\beta_n$ is the homogenisation of a representative of $\Omega_n$.
  Hence $\theta_n \in \uni R \mg$, which finishes the proof.
\end{proof}

\subsection{The universal Weierstrass curve and the proof in the general case} \label{univ_sect_sing}

As in the case of smooth Weierstrass curves, there exists a {\em universal Weierstrass curve} $\uni C$ (as defined in \autoref{const_sect_defn}) and a {\em universal point} $\uni P = \brc*[big]{\pth*[big]{\sh O_{\uni C\sm}(1),x,y,z}} \in \uni C \sm(\uni C \sm)$, corresponding to the identity map on $\uni C\sm$.
Moreover, the universal smooth Weierstrass curve $\uni E$ is the complement of the closed subscheme of $\uni C\sm$ defined by the elliptic discriminant $\Delta$.

\begin{prop}
  Let $n \in \bb Z$.
  Then we have, in $\uni C\sm(\uni C\sm)$,
  \[
    n \uni P = \brc*[big]{\pth*[big]{\sh O_{\uni C\sm}(n^2),\alpha_n,\beta_n,\gamma_n}}.
  \]
\end{prop}

\begin{proof}
  Write $n \uni P = [(\sh L, s_0,s_1,s_2)] \in \uni C\sm(\uni C\sm)$.
  By \autoref{univ_multn} and \autoref{const_groupscheme_cor}, we have $\sh L|_{\uni E} \iso \sh O_{\uni C\sm}(n^2)|_{\uni E}$.
  Note that $\uni C\sm$ is noetherian integral separated regular, therefore the Weil divisor class group is naturally isomorphic to the Picard group.

  Note moreover that $\Delta$ is irreducible and that $\fdiv{\Delta} = V(\Delta)$, so $V(\Delta)$ is a principal prime divisor.
  As $\uni E = \uni C\sm - V(\Delta)$, it follows that the natural map $\Pic C\sm \to \Pic \uni E$ is an isomorphism.
  Hence $\sh L \iso \sh O_{\uni C\sm}(n^2)$, so we assume without loss of generality that $\sh L = \sh O_{\uni C\sm}(n^2)$.

  Next, we note that by the smooth case of \autoref{univ_main}, we have 
  \[
    u\cdot (s_0|_{\uni E}, s_1|_{\uni E}, s_2|_{\uni E}) = (\alpha_n|_{\uni E}, \beta_n|_{\uni E}, \gamma_n|_{\uni E})
  \]
  for a unique unit $u$ of $\uni R = \uni A_\Delta$.
  As $\uni E$ is dense in $\uni C\sm$, we deduce that $u \cdot (s_0, s_1, s_2) = (\alpha_n, \beta_n, \gamma_n)$.
  Note that the units of $\uni R$ are $\pm \Delta^i$ with $i \in \bb Z$.
  If we pull back both triples via the zero section (i.e.~ we evaluate all of the sections at $(0:1:0)$), we see that both $u s_0(0,1,0)$ and $u s_2(0,1,0)$ are zero, and $u s_1(0,1,0) = 1$, as global sections of $\sh O_{\uni A}$, i.e.~elements of $\uni A$.
  Hence we have $s_0(0,1,0) = s_2(0,1,0) = 0$, so as $s_0, s_1, s_2$ generate $\sh L$, we deduce that $s_1(0,1,0)$ is a unit in $\uni A$.
  So $u$ is a unit in $\uni A$ as well, i.e.~equal to $1$ or $-1$.
\end{proof}

\autoref{univ_main} follows from this by the universality of $\uni C$ and $\uni P$.
Moreover, as the polynomials $\alpha_n$, $\beta_n$, $\gamma_n$ were defined as homogenisations of $\Phi_n\Psi_n$, $\Omega_n$, $\Psi_n^3$, it immediately follows that we have the following scheme-theoretic equivalent of \autoref{const_main_cor}.

\begin{cor} \label{univ_classic}
  Let $n \in \bb Z$.
  Let $S$ be a scheme, and let $C$ be a Weierstrass curve over $S$ defined by $W$ as in (\ref{const_eqn_weierstrass}).
  Let $T$ be an $S$-scheme, and let $P \in C(T)$ be of the form $(x:y:1)$, such that $a_1y-3x^2-2a_2x-a_4$, $2y+a_1x+a_3$ generate $\sh O_T$.
  Then we have
  \[
    nP = \pth*[big]{\Phi_n(x,y)\Psi_n(x,y) : \Omega_n(x,y) : \Psi_n^3(x,y)}.
  \]
\end{cor}


\section{Application to a moduli problem} \label{moduli_sect_moduli}

In this section, we use \autoref{univ_classic} to give an explicit description of $Y_1(n)$ over $\bb Z[1/n]$.
These turn out to coincide with the descriptions of $Y_1(n)_{\bb C}$ given for example by \cite{baaziz}.

\subsection{Definitions}

Let $S$ be a scheme.
By \autoref{const_flatlfp}, all smooth Weierstrass curves over a scheme $S$ (together with the section $(0:1:0)$) are also elliptic curves over $S$.
Moreover, we have the following partial converse, treated for example in \cite[Sect.~2.2]{katzmazur}.

\begin{thm}
  Let $E$ be an elliptic curve over $S$.
  Then Zariski locally on the base, $E$ is isomorphic to a smooth Weierstrass curve.
\end{thm}

By \autoref{const_groupscheme}, any elliptic curve admits a unique commutative group scheme structure that has the given zero section as neutral element.

We now make the following definition.

\begin{defn}
  Let $E$ be an elliptic curve over $S$, and let $n$ be a positive integer.
  A {\em $\bb Z/n\bb Z$-embedding} into $E$ is a section $P \in E(S)$ such that $nP = 0$, and such that $P$ has order $n$ in every geometric fibre of $E$ over $S$.
\end{defn}

In \cite[Sect.~1.4]{katzmazur}, this is called an ``\'etale point of exact order $n$'' (which is the same as a ``point of exact order $n$'' if $n$ is invertible in $S$).
This condition on the point $P$ is equivalent to the resulting homomorphism $(\bb Z/n\bb Z)_S \to E$ being a closed immersion, and if $S$ is the spectrum of a field in which $n$ is invertible, then it is even equivalent to being a point of order $n$ in $E(S)$.

Following \cite{katzmazur}, we define the category $\Ell$ of elliptic curves as the category in which the objects are pairs $(E,S)$ with $E$ an elliptic curve over $S$, and in which $\Hom_{\Ell}\pth*[big]{(E,S),(E',S')}$ is the set of pairs $(a,s)$ of morphisms $a \colon E \to E'$ and $s \colon S \to S'$ making the diagram
\[
  \begin{tikzcd}
    E \arrow{r}{a} \arrow{d} & E' \arrow{d} \\
    S \arrow{r}{s} & S'
  \end{tikzcd}
\]
Cartesian, which thereby induce an isomorphism $E \to E' \fp[S'] S$.

Let $n$ be a positive integer.
Then let $\col P(n)$ be the {\em naive $\Gamma_1(n)$} moduli problem $\Ell \to \Set$ attaching to a pair $(E,S)$, the set of $\bb Z/n \bb Z$-embeddings into the elliptic curve $E$ over $S$.
(We call it {\em naive} since we consider, in the terminology of \cite[Sect.~1.4]{katzmazur}, ``\'etale points of exact order $n$'' instead of ``points of exact order $n$''.)
In the remainder of this section, we will study this moduli problem for $n \geq 4$.

\subsection{$\bb Z/n\bb Z$-embeddings and classical division polynomials}

In this section, we state some divisibility properties of division polynomials, and their implications for $\bb Z/n \bb Z$-embeddings.
The notation is as in \autoref{const_sect_generic}.
Moreover, recall that $\uni R = \uni A_{\Delta}$.

\begin{prop} \label{moduli_generate}
  Let $n \in \bb Z - \set{0}$ be a non-zero integer.
  Then $\Phi_n$ and $\Psi_n$ generate $\uni R[X,Y]/(\uni W')$ as an ideal.
\end{prop}

\begin{proof}
  Let $R = \uni R[X,Y]/(\uni W',\Phi_n)$.
  We show that $\Psi_n$ is invertible in this ring, or equivalently, that $\Psi_n$ does not lie in any maximal ideal of $R$.
  So let $\ideal m$ be a maximal ideal of $R$, and let $k$ be an algebraic closure of $R/\ideal m$.
  Let $E$ be the smooth Weierstrass curve corresponding to the canonical morphism $\uni R \to k$, and let $P = (X:Y:1) \in E(k)$.
  Note that the image of any $f \in \uni R[X,Y]/(\uni W')$ in $k$ is $f(P)$.

  Suppose for a contradiction that $\Psi_n(P) = 0$.
  As $X\Psi_n(P)^2 - \Psi_{n+1}(P)\Psi_{n-1}(P) = \Phi_n(P) = 0$, we see that either $\Psi_{n-1}(P) = 0$ or $\Psi_{n+1}(P) = 0$.
  Hence by \autoref{univ_classic}, we have $nP = 0$ and either $(n-1)P = 0$ or $(n+1)P = 0$, which implies that $P = 0$.
  This is a contradiction.
  We deduce that $\Psi_n$ does not lie in any maximal ideal of $R$, as desired.
\end{proof}

Using the formula in \autoref{univ_classic}, we immediately deduce the following.

\begin{cor} \label{moduli_torsion}
  Let $S$ be a scheme, let $E$ be a smooth Weierstrass curve over $S$.
  Moreover, let $P \in E(S)$ be of the form $(a:b:1)$ and let $n$ be an integer.
  Then $nP = 0$ if and only if $\Psi_n(P) = 0$.
\end{cor}

Moreover, using the observation that if $(\sh L, s_0, s_1, s_2) \in E(S)$ is non-zero in every geometric fibre of $E$ over $S$, then $s_2$ generates $\sh L$, we deduce the following.

\begin{cor}
  Let $S$ be a scheme, let $E$ be a smooth Weierstrass curve over $S$.
  Moreover, let $P \in E(S)$ and let $n$ be an integer.
  Then $nP$ is non-zero in every geometric fibre if and only $P$ is of the form $(a:b:1) \in E(S)$ with $a,b \in \sh O_S(S)$, and $\Psi_n(P) \in \sh O_S(S) \mg$.
\end{cor}

This gives us an explicit way of expressing when a section $P$ of a smooth Weierstrass curve $E$ defines a $\bb Z/n \bb Z$-embedding.

\begin{cor} \label{moduli_embedding}
  Let $S$ be a scheme, let $E$ be a smooth Weierstrass curve over $S$.
  Moreover, let $P \in E(S)$ and let $n$ be an integer with $\abs n > 1$.
  Then $P$ is a $\bb Z/n\bb Z$-embedding if and only if $P$ is of the form $(a:b:1)$ with $a,b \in \sh O_S(S)$, $\Psi_n(P) = 0$, and for all divisors $d \neq n$ of $n$, $\Psi_d(P) \in \sh O_S(S) \mg$.
\end{cor}

If we invert $n$ in $\uni R$, we can show more.
First, we show the following.

\begin{lem} \label{moduli_divide}
  Let $n \in \bb Z - \set{0}$, and let $d$ be a divisor of $n$.
  Then $\Psi_d \mid \Psi_n$ in $\uni R[X,Y]/(\uni W')$.
\end{lem}

\begin{proof}
  Let $R = \uni R[X,Y]/(\uni W', \Psi_d)$, and let $E$ be the smooth Weierstrass curve corresponding to the canonical morphism $\uni R \to R$.
  Let $P = (X:Y:1) \in E(R)$.
  Again, note that the image of any $f \in \uni R[X,Y]/(\uni W')$ in $k$ is $f(P)$.
  By \autoref{moduli_torsion}, we now see that $dP = 0$ in $E(R)$, hence $nP = 0$, so $\Psi_n = 0$ in $R$.
  We deduce that $\Psi_d \mid \Psi_n$ in $\uni R[X,Y]/(\uni W')$.
\end{proof}

In the same way as in \autoref{moduli_generate}, we now prove the following proposition.

\begin{prop}
  Let $n \in \bb Z - \set{0}$, and let $d$ be a divisor of $n$.
  Then $\Psi_d$ and $\Psi_n/\Psi_d$ generate $\uni R[1/n,X,Y]/(\uni W')$ as an ideal.
\end{prop}

\begin{proof}
  Let $R = \uni R[1/n,X,Y]/(\uni W',\Psi_d)$.
  We show that $\Psi_n / \Psi_d$ is invertible in $R$, or equivalently, that $\Psi_n / \Psi_d$ does not lie in any maximal ideal of $R$.
  So let $\ideal m$ be a maximal ideal of $R$, and let $k$ be an algebraic closure of $R/\ideal m$.
  Let $E$ be the smooth Weierstrass curve over $k$ corresponding to the canonical morphism $\uni R \to k$, and let $P = (X:Y:1) \in E(k)$.

  Note that $\fdiv \Psi_i = \sum_{P \in E[i]-0} \div P - (i^2-1) \div 0$ for all non-zero integers $i$ invertible in $k$, so $\fdiv \Psi_n / \Psi_d = \sum_{P \in E[n] - E[d]} \div P - (n^2-d^2) \div 0$.
  As $\Psi_d(P) = 0$, it now follows that $(\Psi_n / \Psi_d)(P) \neq 0$ in $k$.
  We deduce that $\Psi_n / \Psi_d$ is not in any maximal ideal of $R$, as desired.
\end{proof}

If we then define the elements $F_n \in \uni R[X,Y]/(\uni W')$ recursively by
\[
  F_n = \left\{
    \begin{array}{ll}
      1 & \text{if }n = 1\\
      \Psi_n \prod_{d\mid n, 0<d<n} F_d^{-1} & \text{if }n \geq 2
    \end{array}
    \right.,
  \]
  then we have the following consequence of \autoref{moduli_embedding}.

  \begin{cor} \label{moduli_embedding_2}
    Let $S$ be a scheme, let $E$ be a Weierstrass curve over $S$.
    Moreover, let $P \in E(S)$ and let $n$ be a non-zero integer that is invertible in $S$.
    Then $P$ is a $\bb Z/n\bb Z$-embedding if and only if $P$ is of the form $(a:b:1)$ with $a,b \in \sh O_S(S)$ and $F_n(P) = 0$.
  \end{cor}

  \subsection{The Tate normal form and representability of $\uni P(n)$}

  We recall the following well-known theorem, and its proof.

  \begin{thm}[Tate normal form]
    Let $S$ be a scheme, let $E$ be an elliptic curve over $S$, and let $P \in E(S)$ be such that $2P \neq 0$ and $3P \neq 0$ in every geometric fibre of $E$ over $S$.
    (For example, if $P$ is a $\bb Z/n \bb Z$-embedding with $n \geq 4$.)
    Then there exist unique $s \in \sh O_S(S)$, $t \in \sh O_S(S)\mg$ such that there exists a unique $S$-isomorphism from $E$ over $S$ to the smooth Weierstrass curve $C$ given by the Weierstrass polynomial
    \[
      y^2z + (1+s)xyz + tyz^2 - x^3 - tx^2z,
    \]
    identifying $P \in E(S)$ and $(0:0:1) \in C(S)$.
  \end{thm}

  \begin{proof}
    We work affine locally on $S$.
    Since we require uniqueness of $s$ and $t$ and of the isomorphism, we can glue the obtained local isomorphisms afterwards to get the desired isomorphism.
    Hence we may assume that $E/S$ is in fact a smooth Weierstrass curve, and that $S = \Spec R$ for some ring $R$.

    For the existence, suppose that $E$ is defined by the Weierstrass equation
    \[
      y^2z + a_1xyz + a_3yz^2 - x^3 - a_2x^2z - a_4xz^2 - a_6z^3.
    \]
    Then note that $P \neq 0$ on all geometric fibres, so $P$ is of the form $(x:y:1)$.
    So after a suitable translation, we may assume that $P = (0:0:1)$ and hence $a_6 = 0$.
    As $2P \neq 0$ on all geometric fibres, it then follows that $a_3 \in R\mg$.
    Hence after composing with a suitable shearing parallel to the line $y = 0$, we may assume that $a_4$ is zero.
    Finally, as $3P \neq 0$ on all geometric fibres, we have $a_2 \in R\mg$, hence after a suitable scaling, we find an isomorphism between $E$ and a Weierstrass curve $C$ over $S$ of the desired form.

    For uniqueness, note that by \cite[Sect.~2.2]{katzmazur}, any isomorphism between Weierstrass curves over $S$ corresponds to
    \[
      \pth*[big]{\sh O(1),\alpha^2x+az,\alpha^3y+bx+cz,z} \in \bb P^2_S(\bb P^2_S)
    \]
    on the ambient $\bb P^2_S$ with $\alpha \in R\mg,\ a,b,c \in R$.
    Let $s,s' \in R$, $t,t' \in R\mg$, and let $C$, $C'$ be the Weierstrass curves over $S$ defined by the polynomials
    \[
      y^2z + (1+s)xyz + tyz^2 - x^3 - tx^2z
      \qquad\text{and}\qquad
      y^2z + (1+s')xyz + t'yz^2 - x^3 - t'x^2z,
    \]
    respectively.
    Suppose that $i \colon C \to C'$ is an isomorphism sending $(0:0:1)$ to $(0:0:1)$ and that it is given on the ambient $\bb P^2_S$ by
    \[
      \pth*[big]{\sh O(1),\alpha^2x+az,\alpha^3y+bx+cz,z} \in \bb P^2_S(\bb P^2_S)
    \]
    for $\alpha \in \sh O_S(S)\mg$, $a,b,c \in \sh O_S(S)$.
    Then $a = c = 0$ as $i(0:0:1) = (0:0:1)$, $\alpha = 1$ as the $yz^2$- and $x^2z$-coefficients sum to $0$, and $b = 0$ as the $xz^2$-coefficient is zero.
    Hence $s = s'$, $t = t'$, and $i$ is the identity, which shows the uniqueness.
  \end{proof}

  In order to state the corollaries, we introduce some more notation.
  Let $\delta = \Delta(1+s,t,t,0,0) \in \uni S$.
  Moreover, for any integer $n$, let $\psi_n = \Psi_n(0,0) \in \bb Z[s,t]$ and $f_n = F_n(0,0) \in \bb Z[s,t]$.
  Now we have the following consequences of \autoref{moduli_embedding} and \autoref{moduli_embedding_2}.

  \begin{cor}
    Let $n$ be a positive integer.
    The moduli problem $\uni P(n)$ is representable (over $\bb Z$) by 
    \[
      Y_1(n)^{\text{naive}} = \Spec \bb Z[s,t,\delta^{-1},p_n^{-1}]/(\psi_n).
    \]
    Here, $p_n = \prod_{d\mid n, 0<d<n} \psi_d$.
    The universal elliptic curve over $Y_1(n)^{\text{naive}}$ is the Weierstrass curve given by the Weierstrass polynomial
    \[
      y^2z + (1+s)xyz + tyz^2 - x^3 - tx^2z,
    \]
    together with the point $P = (0:0:1)$.
  \end{cor}

  \begin{cor}
    Let $n$ be a positive integer.
    The moduli problem $\uni P(n)$ is representable over $\bb Z[1/n]$ by 
    \[
      Y_1(n)_{\bb Z[1/n]} = \Spec \bb Z[s,t,\delta^{-1},\tfrac1n]/(f_n).
    \]
    The universal elliptic curve over $Y_1(n)$ is the Weierstrass curve given by the Weierstrass polynomial
    \[
      y^2z + (1+s)xyz + tyz^2 - x^3 - tx^2z,
    \]
    together with the point $P = (0:0:1)$.
  \end{cor}

  \bibliographystyle{abbrvnat}
  \bibliography{divpol_main}
\end{document}